\documentclass[12pt,reqno]{amsart}
\usepackage{amssymb}
\usepackage{color}
\usepackage[all]{xy}
\usepackage{amsmath}

 \rm

\renewcommand{\(}{\left(}
\renewcommand{\)}{\right)}
\renewcommand{\[}{\left[}
\renewcommand{\]}{\right]}

\renewcommand{\>}{\rangle}

\renewcommand{\bar}{\overline}
\newcommand{\abs}[1]{\left\lvert#1\right\rvert}
\newcommand{\norm}[1]{\left\lVert#1\right\rVert}
\newcommand{\st}{\:|\:}

\newcommand{\C}{{\mathbb{C}}}
\newcommand{\R}{{\mathbb{R}}}
\newcommand{\Z}{{\mathbb{Z}}}
\newcommand{\N}{{\mathbb{N}}}

\renewcommand{\phi}{\varphi}

\renewcommand{\Re}{{\mathrm{Re}}}

\newcommand{\supp}{{\mathrm{supp}}}

\theoremstyle{plain}
\newtheorem{thm}{Theorem}[section]
\newtheorem{lem}[thm]{Lemma}

\newtheorem{conj}[thm]{Conjecture}

\theoremstyle{definition}

\theoremstyle{remark}
\newtheorem{rem}[thm]{Remark}

\title[Gaussian decay for the Harmonic oscillator]
{Gaussian decay for the Harmonic oscillator}

\author{Manish Chaurasia}

\address{Department of Mathematics, 
Visvesvaraya National Institute of Technology (VNIT), 
Nagpur 440010, India}
\email{manish2700c@gmail.com}

\begin{document}

\begin{abstract}
We consider the Schrödinger equation associated with the harmonic oscillator and show that if the initial data and its Fourier transform are dominated by Gaussian functions of widths $a>0$ and $b>0$, respectively, satisfying $ab<1$, then the evolved solution and its Fourier transform are dominated by a Gaussian of width $\frac{1}{2}\(\frac{1}{a}+\frac{1}{b}-
\sqrt{\(\frac{1}{a}+\frac{1}{b}\)^2-4}\),$ for all times except for a discrete set, and for all times in one dimension. In the one-dimensional case, we prove that these estimates are sharp. Moreover, for a more restrictive class of initial data, we establish sharper time-dependent Gaussian bounds.
\end{abstract}

\maketitle

\section{Introduction}
The {\em Fourier transform} of $f \in L^1(\R^d)$ is defined by
\begin{equation*}
\mathcal{F}f(\xi) = 
\widehat f(\xi) =  \int_{\R^d} f(x) e^{-2\pi i\xi\cdot x} \, dx.
\end{equation*}
For $a,b>0$, let $g_a(x) = e^{-a\pi\abs{x}^2}$ be a Gaussian of width $a$, 
and let
\begin{equation*}
E^d(a,b) = \{ f \in L^1(\R^d) \st
\abs{f(x)} \le C g_a(x) \text{ and } \abs{\mathcal{F}f(\xi)} \le C g_b(\xi)
\quad\text{for some $C \in \R$} \}.
\end{equation*}
These spaces are called {\it Hardy class}, see 
\cite{Folland-Sitaram, Garg2009, Chaurasia-2024}.
A fundamental result concerning these spaces is Hardy's classical uncertainty principle \cite{Hardy}.
\begin{thm}\label{T:hardy}
If $ab > 1$ then $E^d(a,b) = \{0\}$.  
If $ab=1$ then $E^d(a,b) = \C g_a$.
If $ab < 1$ then $\dim E^d(a,b) = \infty$.
\end{thm}
The Hardy uncertainty principle was first studied in the context of evolution equations 
by Escauriaza-Kenig-Ponce-Vega in \cite{Escauriaza.et.al.-2006}. Subsequently, 
they published a series of papers in this direction, establishing several remarkable 
estimates for solutions of evolution equations; see, for example, 
\cite{Cowling.et.al.-2010, Escauriaza.et.al.-2008, 
 Escauriaza.et.al.-2016, Escauriaza.et.al.-2010}. 
In \cite{Cowling.et.al.-2010},  
Cowling-Escauriaza-Kenig-Ponce-Vega 
provided a real-variable proof of the Hardy uncertainty principle. 
Related studies for Schr\"odinger equations associated with electromagnetic 
Hamiltonians and related operators can be found in 
\cite{Barcelo.et.al.-2013, Cassano-Fanelli-2015, Cassano-Fanelli-2017}.


Here, we are interested in the following Cauchy problem for the Schr\"odinger equation
associated with harmonic oscillator.
\begin{equation}\label{E:HSE}
\begin{aligned}
\frac{1}{i} \frac{\partial u(x,t)}{\partial t} 
=&\; (-\Delta+4\pi^2\abs{x}^2) u(x,t)\\
 u(x,0) =&\; u_0(x)\in E^d(a,b)
\end{aligned}
.
\end{equation}
For $n\in\N$, we define the $n$-th Hermite function by
\begin{equation*}
h_n(x) = \frac{(-1)^n}{\sqrt{2^n\pi^{1/2}n!}} e^{x^2/2} 
\frac{d^n}{dx^n} e^{-x^2}.
\end{equation*}
The Hermite functions are the eigenfunctions of the harmonic oscillator. 
It is therefore natural to expect that the theory of Hermite functions plays 
a central role in understanding the above problem.
Indeed, in \cite{Vemuri2008hermite}, Vemuri proved that 
if 
$u_0\in E^1(\tanh 2\alpha, \tanh 2\alpha)$, for $\alpha>0$, then 
\begin{equation*}\label{E:Vemuri-estimate}
\abs{\<u_0, h_n\>}\lesssim n^{-1/4} e^{-n\alpha}.
\end{equation*}
 Certain higher-dimensional analogues of this result were proved in \cite{Garg2009}.
 Also, a refinement of the result was subsequently obtained in \cite{Chaurasia-2024}. 
 Using this estimate, it was shown in \cite{Vemuri2008hermite} that 
the evolution $u(x,t)$ of $u_0$ by (\ref{E:HSE}) satisfies
$$u(x,t)\in E^1(\tanh \gamma, \tanh \gamma),$$ 
for $\gamma<\alpha$. However, the author conjectured that
the above belonging should hold for $\gamma=\alpha$.
In the recent past, Kulikov-Oliveira-Ramos made substantial progress 
on this conjecture, see \cite{Kulikov.et.al-2024}. 
They proved the following result.
\begin{thm}\label{T:KOR}
Let $u(x,t)$ be the solution of the Cauchy problem (\ref{E:HSE}) with 
the initial data $u_0$.
If $u_0\in E^d(\tanh 2\alpha, \tanh 2\alpha)$, for some $\alpha>0$,
then $u(x,t)\in  E^d(\tanh \alpha, \tanh \alpha)$ whenever
$t\notin \{\frac{1}{16}+\frac{k}{8}, k \in \Z \}$.
\end{thm}
Recently, Radchenko and Ramos settled Vemuri's conjecture
in dimension one, see \cite{Radchenko-Ramos-2025}. Their
result read as follows.
\begin{thm}\label{T:Main-Radchenko-Ramos}
Let $d=1$. 
Let $u(x,t)$ be the solution of the Cauchy problem (\ref{E:HSE}) with 
the initial data $u_0$.
If $u_0\in E^1(\tanh 2\alpha, \tanh 2\alpha)$, for some $\alpha>0$,
then $u(x,t)\in  E^1(\tanh \alpha, \tanh \alpha)$. 
\end{thm}
In their proof, along with the estimate (\ref{E:Vemuri-estimate}), 
they established and used the following result. 
\begin{thm}\label{T:Radchenko-Ramos}
Let $\kappa>0$ and $\beta\in\R$. Then, for any $y>0$
\begin{equation}\label{E:WHE}
\sum_{n\ge1} \frac{e^{-\kappa n y}}{n^{\beta}} \abs{h_n(x)}^{\kappa}
\lesssim_{y,\kappa,\beta} \abs{x}^{1-\frac{\kappa}{2}-2\beta} 
e^{-\kappa x^2 \tanh(y)/2}, \quad \forall x\in\R\setminus [-1,1],
\end{equation}
and in addition if $\frac{\kappa}{2}+2\beta \ge 1$, then the estimate 
holds for all $x\in\R\setminus \{0\}$. Moreover, the estimate (\ref{E:WHE})
is sharp. 
\end{thm}
Analogues results of Theorem \ref{T:Main-Radchenko-Ramos} 
and \ref{T:Radchenko-Ramos} for certain weighted Hardy class can be found in 
\cite{Achar2026Hermite}.
In this paper, we are interested in understanding
 the behavior of solutions to \eqref{E:HSE} 
corresponding to initial data $u_0\in E^d(a,b)$ when $ab<1$. 

Let $ab<1$. We denote $m=\min\{a,b\}$. First, we note 
that if $u_0\in E^d(a,b)$, then $u_0\in E^d(m,m)$. Therefore,
Theorem \ref{T:KOR} and \ref{T:Main-Radchenko-Ramos} imply that
\begin{equation*}
u(x,t) \in E^d\(\frac{1-\sqrt{1-m^2}}{m}, \frac{1-\sqrt{1-m^2}}{m}\)
\end{equation*}
for suitable set of times. This naturally raises the question of whether a sharper estimate 
for $u(x,t)$ can be obtained. More precisely, 
does there exists a function $\omega(a,b)$ such that 
$$\frac{1-\sqrt{1-m^2}}{m}<\omega(a,b)$$
and 
$$u(x,t)\in E^d(\omega(a,b), \omega(a,b))?$$ 
Here, we answer this question affirmatively.

\subsection{Statements of the main results.}
 In order to state the results, we first
define a few quantities.
From now on, we shall always assume that $ab<1$, unless stated otherwise. 
In that case, we have 
\begin{equation}\label{E:criteria}
\frac{1}{a} + \frac{1}{b} >2.
\end{equation}
Therefore, we can define the following parameters, which will be crucial for us:
\begin{equation}\label{E:v-function}
\begin{aligned}
\omega(a,b) =&\; \frac{1}{2}\(\frac{1}{a}+\frac{1}{b}-
\sqrt{\(\frac{1}{a}+\frac{1}{b}\)^2-4}\),\\
t_c(a,b) =&\; \frac{1}{8\pi } \sqrt{\frac{ab}{1-ab}}
\(\frac{1}{a}-\frac{1}{b}+
\sqrt{\(\frac{1}{a}+\frac{1}{b}\)^2-4}\), \quad\text{and}\\
l_{\pm} (a,b) =&\; \frac{-\sqrt{ab(1+\sqrt{ab})}\pm ab\sqrt{2\sqrt{ab}}}
{b\sqrt{1-\sqrt{ab}}},
\end{aligned}
\end{equation}
where $l_+$, and $l_-$ are right hand side quantities with $+$, and $-$ sign 
respectively. 
Also, we define the following functions:
\begin{equation}\label{E:omega-lambda}
\Omega_{a,b}(t) = \frac{2a b (1+16\pi^2t^2)}
{a+8\pi\sqrt{ab(1-ab)}t+16\pi^2bt^2}, \quad\text{and}\quad
\Lambda_{a,b}(t) = \frac{2\sqrt{ab}(1+16\pi^2t^2)}{a+16\pi^2bt^2} 
\quad t\in\R.
\end{equation}
Our first result is the following one, which generalizes Theorem \ref{T:KOR}.
\begin{thm}\label{T:Asym-evolution}
Let $a,b>0$, and $ab<1$. 
Let $\omega$ and $t_c$ be as in equation \eqref{E:v-function}.
Let $u(x,t)$ be the solution of the Cauchy problem (\ref{E:HSE}) with 
the initial data $u_0$.
If $u_0\in E^d(a,b)$, then 
$$
u(x,t)\in E^d(\omega(a,b),\omega(a,b)), 
$$
whenever 
$t\notin \Big\{\frac{\arctan (4\pi t_c(a,b))}{4\pi}
+\frac{k}{8}\st k \in 2\Z \Big\}
\cup\Big\{\frac{\arctan (4\pi t_c(b,a))}{4\pi}
+\frac{k}{8}\st k \in \Z\setminus2\Z \Big\}$.
\end{thm}
It is shown in \cite{Chaurasia-2025} that if $u_0\in E^1(a,b)$, then 
\begin{equation}\label{E:Chaurasia-estimate}
\abs{\<u_0, h_n\>}\lesssim n^{-1/4} \(\frac{a+b-2ab}{a+b+2ab}\)^{n/4}.
\end{equation}
We shall use the above estimate and Theorem \ref{T:Radchenko-Ramos}
to establish the following generalization of the 
Theorem \ref{T:Main-Radchenko-Ramos}.
\begin{thm}\label{T:1d-evol}
Let $d=1$. Let $a,b>0$, and $ab<1$. 
Let $\omega$ be as in equation \eqref{E:v-function}.
Let $u(x,t)$ be the solution of the Cauchy problem 
(\ref{E:HSE}) with the initial data $u_0$.
If $u_0\in E^1(a,b)$, then 
$$
u(x,t)\in E^1(\omega(a,b),\omega(a,b)) \quad t\in\R.
$$
Moreover, the belonging is sharp.
\end{thm}
\begin{rem}
Note that $\frac{1-\sqrt{1-m^2}}{m}<\omega(a,b)$.
\end{rem}
Clearly, uniform Gaussian bounds for solutions do not capture their full behavior. 
This motivates the study of time-dependent Gaussian bounds for the evolution. 
In \cite{Kulikov.et.al-2024}, this problem was considered for initial data in 
$E^1(\tanh 2\alpha,\tanh 2\alpha)$, for $\alpha>0$. 
Here, we study time-dependent 
Gaussian bounds for solutions of \eqref{E:HSE} with initial data in $E^1(a,b)$
satisfying a suitable compatibility condition with the following Gaussian.
\begin{equation*}
G_{a,b}(x) =  e^{- \frac{\pi}{b}\(ab+i\sqrt{ab(1-ab)}\)\abs{x}^2}.
\end{equation*}
This Gaussian was used in \cite{Chaurasia-2025} to saturate the 
estimate (\ref{E:Chaurasia-estimate}). 
Indeed, we shall prove the 
following result.
\begin{thm}\label{T:time-bound}
Let $a,b>0$, and $ab<1$. 
Let $\Omega_{a,b}$ and $\Lambda_{a,b}$ be as in \eqref{E:omega-lambda}.
Let 
$$L_{a,b} = \Big\{\(v,w\)\subset\R^2\st 0\le v, w, 0\le vw \le 1, 
aw=bv\Big\}.$$
 Let $u(x,t)$ be the solution of the Cauchy problem 
(\ref{E:HSE}) with the initial data $u_0$, and 
\begin{equation*}
u_0(x) = \int_{L_{a,b}} G_{v,w}(x) d\mu(v,w),
\end{equation*}
for some measure $\mu$. If $u_0\in E^1(a,b)$, then
$$
u(x,t) \in E^1\(\Xi_{a,b}(s),\Xi_{a,b}(s)\),
$$
where $s=-\frac{\tan4\pi t}{4\pi}$, $t\in\R$, and
\begin{equation*}
\Xi_{a,b}(y) = 
\begin{cases}
\Omega_{a,b}(y)  \quad y\in\R\setminus (l_-(a,b)+\frac{k}{4}, l_+(a,b)+\frac{k}{4})\\ 
\Lambda_{a,b}(y)  \quad y\in(l_-(a,b)+\frac{k}{4}, l_+(a,b)+\frac{k}{4})
\end{cases}
k\in\Z.
\end{equation*}
\end{thm}
The techniques used in the proofs of the above results are available in the literature,
 specifically in \cite{Kulikov.et.al-2024, Radchenko-Ramos-2025, Cassano-Fanelli-2017}. 
 To apply these ideas, several additional ingredients are required, which we introduce as
  needed.
\section{The Proofs}
Let $a,b>0$, and $ab<1$
\subsection{Proof of Theorem \ref{T:Asym-evolution}.}
Let $\bold{h}_{\alpha}(x)$, $x\in\R^d$, $\alpha\in\N^d_0$,
denotes the normalized Hermite function of order $\alpha$, 
which is defined as the tensor product of one dimensional 
normalized Hermite functions.
We first note the translation between the harmonic oscillator
and the linear Schr\"odinger equation from \cite{Kulikov.et.al-2024}:
\begin{equation}\label{E:harmonic-to-free}
e^{it\Delta}u_0(x) = (1+16\pi^2t^2)^{-d/4} 
\exp{\[\frac{4\pi^2it}{1+16\pi^2 t^2}\abs{x}^2\]}\cdot
u\(\frac{x}{\sqrt{1+16\pi^2 t^2}}, \frac{\arctan(-4\pi t)}{4\pi}\),
\end{equation}
which was achieved by using the following formula established 
in \cite[Lemma 11]{Goncalves-2019}
\begin{equation*}
e^{it\Delta}(\bold{h}_{\alpha})(x) =
(1+4\pi i t)^{-d/2}\(\frac{1-4\pi i t}{1+4\pi i t}\)^{\abs{\alpha}/2}
\exp{\[\frac{4\pi^2it}{1+16\pi^2 t^2}\abs{x}^2\]}\cdot
 \bold{h}_{\alpha}\(\frac{x}{\sqrt{1+16\pi^2 t^2}}\).
\end{equation*}
Now to analyze the free solutions, we note from
\cite{Cowling.et.al.-2010} that, if $v_1$ is a solution of the problem
\begin{equation*}
\begin{cases}
\begin{aligned}
\frac{1}{i} \frac{\partial v_1(x,t)}{\partial t} 
=&\; -\Delta v_1(x,t)\\
 v_1(x,0) =&\; g(x)
\end{aligned}
,
\end{cases}
\end{equation*}
then the function
\begin{equation}\label{E:v1-v2}
v_2(x,t) = (it)^{-d/2} e^{-\frac{\abs{x}^2}{4it}} \bar{v}_1(x/t, 1/t-1)
\end{equation}
verifies
\begin{equation*}
\begin{cases}
\begin{aligned}
\frac{1}{i} \frac{\partial v_2(x,t)}{\partial t} 
=&\; -\Delta v_2(x,t)\\
 v_2(x,0) =&\; (4\pi)^{-d/2} e^{\frac{\abs{x}^2}{4i}} 
 \bar{\hat{g}}(x/4\pi)\\
 v_2(x,1) =&\; i^{-d/2} e^{-\frac{\abs{x}^2}{4i}} 
 \bar{g}(x)
\end{aligned}
.
\end{cases}
\end{equation*}
Hence, our problem reduces to analyze the functions $v_2(x,t)$.

Let $g(x) = u_0\(\frac{x}{2\sqrt{\pi}}\)$. 
Since, $u_0\in E^d(a,b)$,
therefore, we obtain
$v_2(x,0) \in L^2(e^{\frac{b-\epsilon}{4}\abs{x}^2}dx)$, and
$v_2(x,1) \in L^2(e^{\frac{a-\epsilon}{4}\abs{x}^2}dx)$,
for every $\epsilon>0$.
The key to further analysis is Theorem 3 from 
\cite{Escauriaza.et.al.-2010} by Escauriaza-Kenig-Ponce-Vega; 
we require its full strength in this context.
\begin{lem}\label{L:L2-evolution}
Assume that $y\in C([0,1], L^2(\R^d))$ verifies 
\begin{equation*}
\partial_t y = i \Delta y, \quad \text{in}\,\, \R^d\times [0,1],
\end{equation*}
and $\alpha\beta\ge 4$. then 
\begin{equation*}
\sup_{[0,1]}
\norm{e^{\Sigma_{\alpha,\beta}(t)\abs{x}^2}y(t)}_{L^2(\R^d)}
\lesssim_{\alpha,\beta} 
\norm{e^{\frac{\abs{x}^2}{\beta^2}}y(0)}_{L^2(\R^d)} 
+ \norm{e^{\frac{\abs{x}^2}{\alpha^2}}y(1)}_{L^2(\R^d)},
\end{equation*}
where 
\begin{equation*}
\Sigma_{\alpha,\beta}(t) = \frac{R\alpha\beta}
{2[(\alpha t+\beta(1-t))^2+R^2(\alpha t-\beta(1-t))^2]},
\end{equation*}
$R$ is the smallest root of the equation
\begin{equation*}
\frac{1}{\alpha\beta} = \frac{R}{2(1+R^2)}.
\end{equation*}
\end{lem}
Write $a_{\epsilon} = a - \epsilon$, and 
$b_{\epsilon} = b - \epsilon$. 
We have $ab<1$, therefore, for 
$\alpha = \frac{2}{\sqrt{a_{\epsilon}}}$, and 
$\beta = \frac{2}{\sqrt{b_{\epsilon}}}$, we use above lemma
for $v_2(x,t)$, and obtain that
\begin{equation*}
\norm{e^{A_{a,b,\epsilon}(t)\abs{x}^2}v_2(x,t)}_{L^2(\R^d)}
< +\infty,
\end{equation*}
where 
\begin{equation*}
A_{a,b,\epsilon}(t) = 
\Sigma_{\frac{2}{\sqrt{a_{\epsilon}}}, 
\frac{2}{\sqrt{b_{\epsilon}}}}(t)
= \frac{\sqrt{a_{\epsilon}b_{\epsilon}}R_{a,b,\epsilon}}
{2[(\sqrt{a_{\epsilon}}\, t+\sqrt{b_{\epsilon}}(1-t))^2+
R_{a,b,\epsilon}^2(\sqrt{a_{\epsilon}}\, t-\sqrt{b_{\epsilon}}(1-t))^2]},
\end{equation*}
and
\begin{equation*}
R_{a,b,\epsilon}= \frac{1-\sqrt{1-a_{\epsilon}b_{\epsilon}}}
{\sqrt{a_{\epsilon}b_{\epsilon}}}.
\end{equation*}
By the relation (\ref{E:v1-v2}), we get that
\begin{equation*}
\norm{e^{B_{a,b,\epsilon}(t)\abs{x}^2}v_1(x,t)}_{L^2(\R^d)}
< +\infty,
\end{equation*}
where $B_{a,b,\epsilon}$ is defined by the following relation:
\begin{equation*}
(1+t)^2B_{a,b,\epsilon}(t) = A_{a,b,\epsilon}\(\frac{1}{t+1}\).
\end{equation*} 
Finally, we return to our original functions by the
relation $e^{it\Delta}u_0(x,t)= v_1(2\sqrt{\pi}x, 4\pi t)$, and 
in that case, we acquire 
\begin{equation*}
\norm{e^{4\pi B_{a,b,\epsilon}(4\pi t)\abs{x}^2}
e^{it\Delta}u_0(x,t)}_{L^2(\R^d)}< +\infty,
\end{equation*}
for all $t>0$.
Then, using the translation relation (\ref{E:harmonic-to-free})
of harmonic oscillator and linear Schr\"odinger equation, 
it follows that
\begin{equation}\label{E:L2-e-estimate}
\norm{e^{4\pi (1+16\pi^2t^2)B_{a,b,\epsilon}(4\pi t)\abs{x}^2}
u\(x, \frac{\arctan(-4\pi t)}{4\pi}\)}_{L^2(\R^d)}< +\infty,
\end{equation}
for all $t>0$.

To proceed, we analyze the function
$$\Omega_{a,b,\epsilon}(t) = 4\pi (1+16\pi^2t^2)B_{a,b,\epsilon}(4\pi t).$$
That is, we look at the function
\begin{equation*}
\Omega_{a,b,\epsilon}(t) = 
\frac{2\pi\sqrt{a_{\epsilon}b_{\epsilon}}R_{a,b,\epsilon}(1+16\pi^2t^2)}
{[(\sqrt{a_{\epsilon}}+4\pi\sqrt{b_{\epsilon}}t)^2+
R_{\epsilon}^2(\sqrt{a_{\epsilon}}-4\pi\sqrt{b_{\epsilon}}t)^2]},
\end{equation*}
and show that 
\begin{equation}\label{E:critical-variance}
\min_t \Omega_{a,b,\epsilon}(t) = \Omega_{a,b,\epsilon}
(t_c(a_{\epsilon}, b_{\epsilon})) = 
\omega(a_{\epsilon}, b_{\epsilon}).
\end{equation}
To do so, we write $s(a,b)=a+b$, $p(a,b)=ab$ and 
$r(a,b)=\sqrt{s^2-4p^2}$, and define
\begin{equation*}
P_{a,b}(t) = 16\pi^2 b \(2ap-s+r\)t^2-8\pi\sqrt{p(1-p)}
\(s-r\)t+a\(2pb-s+r\).
\end{equation*}
Observe that 
\begin{equation*}
\(2ap-s+r\)\(2pb-s+r\) = (1-p) \(s-r\)^2.
\end{equation*}
Therefore,
\begin{equation*}
P_{a,b}(t) = \(t-\frac{\sqrt{p(1-p)}\(s-r\)}{4\pi b \(2ap-s+r\)}\)^2
\end{equation*}
A direct calculation shows
\begin{equation*}
b(2ap-s+r)(b-a+r)=2p(1-p)(s-r).
\end{equation*}
Thus, from equation (\ref{E:v-function}) we obtain
\begin{equation*}
P_{a,b}(t) = \(t-t_c\)^2\ge 0.
\end{equation*}
At this point, we see that
\begin{equation*}
\Omega_{a,b,\epsilon}(t) - 
\omega(a_{\epsilon}, b_{\epsilon}) = 
\frac{P_{a_{\epsilon}, b_{\epsilon}}(t)}{2a_{\epsilon}b_{\epsilon}
(a_{\epsilon}+8\pi\sqrt{a_{\epsilon}b_{\epsilon}(1-a_{\epsilon}b_{\epsilon})}t
+16\pi^2b_{\epsilon}t^2)}\ge 0,
\end{equation*}
and
\begin{equation*}
\Omega_{a,b,\epsilon}(t_c) = \omega(a_{\epsilon}, b_{\epsilon}),
\end{equation*}
which together justifies our claim (\ref{E:critical-variance}).

From equation (\ref{E:v-function}) and (\ref{E:critical-variance}), we have
\begin{equation*}
\begin{aligned}
\Omega_{a,b}(t) &=\; \lim_{\epsilon\rightarrow 0}
\Omega_{a,b,\epsilon}(t) \quad \text{and}\\
\Omega_{a,b}(t_c) &=\; \omega(a, b).
\end{aligned}
\end{equation*}
Hence, equation (\ref{E:L2-e-estimate}) implies that
\begin{equation*}
\norm{e^{\theta\abs{x}^2}
u(x, s)}_{L^2(\R^d)}< +\infty,
\end{equation*}
for all $\theta<\Omega_{a,b}\(-\frac{\tan 4\pi s}{4\pi}\)$.
We have that 
$\Omega_{a,b}\(-\frac{\tan 4\pi s}{4\pi}\)>\omega(a, b)$ for all
$s\in (-\frac{\arctan (4\pi t_c(a,b))}{4\pi}, 0)$.
Therefore, there exists $\theta(s)>\omega(a, b)$ such that
\begin{equation}\label{E:L2-g-estimate}
\norm{e^{\theta(s)\abs{x}^2} u(x, s)}_{L^2(\R^d)}< +\infty.
\end{equation}

To extend this estimate to the claimed time sets, we note that 
the $\omega$ is a symmetric function in $a$ and $b$, i.e, 
$\omega(a,b) = \omega(b,a)$, however, $t_c$ is not a symmetric function
for all $a$ and $b$. Also,
if 
$\Phi(x,t)$ is a solution of problem (\ref{E:HSE}) with the initial data
$\mathcal{F}^k u_0$, for $k\in\Z$, then 
\begin{equation*}
\abs{\Phi(x,t)} = \abs{u(x,t-k/8)}.
\end{equation*}
Let us denote 
\begin{equation*}
K_e = \Big\{\frac{\arctan (4\pi t_c(a,b))}{4\pi}
+\frac{k}{8}\st k \in 2\Z \Big\}, \quad \text{and}\quad
K_o = \Big\{\frac{\arctan (4\pi t_c(b,a))}{4\pi}
+\frac{k}{8}\st k \in \Z\setminus2\Z \Big\}.
\end{equation*}
With these, we see two cases:\\
Case I: When $k\in 2\Z$ and $u_0\in E^d(a,b)$. In this case, we see that
$\mathcal{F}^k u_0$ is in $E^d(a,b)$. Then the estimate 
(\ref{E:L2-g-estimate}) holds for all 
$s\in\R\setminus K_e$.\\
Case II: When $k\in \Z\setminus2\Z$ and $u_0\in E^d(a,b)$. 
In this case, we see that $\mathcal{F}^k u_0$ is in $E^d(b,a)$. 
Then the estimate (\ref{E:L2-g-estimate}) holds for all 
$s\in\R\setminus K_o$.\\
Altogether, we have that for all 
$s\in\R\setminus K_e\cup K_o$ 
there exists $\theta(s)>\omega(a,b)$ so that (\ref{E:L2-g-estimate}) holds.

The further analysis to prove the pointwise estimate for $u(x,s)$
proceeds as in \cite{Kulikov.et.al-2024}. We mainly require the following 
result from \cite{Kulikov.et.al-2024}.
\begin{thm}\label{T:Lp-to-Linfty}
If the function $f:\R^d\rightarrow \C$ is such that $e^{a\pi\abs{x}^{\alpha}} 
f\in L^p$ and $e^{a\pi\abs{x}^{\beta}}\hat{f}\in L^q$ for some 
$a,\alpha,\beta>0$ and $p,q\ge 1$, then for all $\epsilon>0$ we have 
$\abs{f(x)}\lesssim e^{-(1-\epsilon)a\pi\abs{x}^{\alpha}}$.
\end{thm}
In fact, recently, in \cite{Saucedo_Tikhonov-2024}, 
Saucedo-Tikhonov proved that the $\epsilon$-loss
in this result is merely of polynomial order. 

In our situation, for $\Theta(s) = \min\{\theta(s),\theta(s-1/8)\}$, we have
\begin{equation*}
e^{\Theta(s)\pi \abs{x}^2} u(x,s), 
e^{\Theta(s)\pi \abs{\xi}^2} \hat{u}(\xi,s) \in L^2,
\end{equation*}
because $s-1/8\notin K_e\cup K_o$ 
when 
$s\notin K_e\cup K_o$,
and 
$\hat{u}(x,s) = u(x,s-1/8)$.
Therefore, we obtain from Theorem \ref{T:Lp-to-Linfty}, that 
\begin{equation*}
u(x,s) \in E^d\(\Theta(s)-\epsilon,\Theta(s)-\epsilon\).
\end{equation*} 
Finally, by making $\epsilon>0$ arbitrary small, we conclude
\begin{equation*}
u(x,s) \in E^d\(\omega(a,b),\omega(a,b)\),
\end{equation*} 
when 
$s\notin K_e\cup K_o$. 
\qed

\subsection{Proof of Theorem \ref{T:1d-evol}.}
We first show that the claimed estimates are sharp in the sense that 
there exists a function which saturates the estimates. To do so,
we need to define a few parameters. 
The equation (\ref{E:criteria}) allows us to define
\begin{equation*}
A(a,b) = \sqrt{\frac{a+b-2ab}{a+b+2ab}}.
\end{equation*}
Observe that 
\begin{equation*}
(a+b-2ab)(a+b+2ab)>(a-b)^2.
\end{equation*}
Therefore, there exists a unique parameter $\tau\in (-\frac{1}{4}, \frac{1}{4})$
such that
\begin{equation*}
\sin2\pi\tau = \frac{a-b}{\sqrt{(a+b-2ab)(a+b+2ab)}}.
\end{equation*}
Choose a branch $\sqrt{}$ of the square root that is defined on the right
half plane and is positive on the positive real line.  Define
$$
u(x,t) =
\frac{e^{\pi i(2t+\pi/8)}}{\sqrt{1-iAe^{4\pi i(2t-\tau/2)}}}
\exp\[-\pi \(\frac{1+iAe^{4\pi i(2t-\tau/2)}}{1-iAe^{4\pi i(2t-\tau/2)}}\) \abs{x}^2\].
$$
Then $u(x,t)$ is a solution of (\ref{E:HSE}).
Observe that
\begin{equation*}
\Re\(\frac{1+iAe^{-2\pi i\tau}}{1-iAe^{-2\pi i\tau}}\) = a 
\quad \text{and} \quad
\Re\(\frac{1-iAe^{-2\pi i\tau}}{1+iAe^{-2\pi i\tau}}\) = b.
\end{equation*}
Also,
\begin{equation}\label{E:A-def}
\frac{1-A}{1+A} 
= \omega(a,b).
\end{equation}
Therefore
\begin{equation*}
\begin{aligned}
\abs{u(x,0)}
=   & \; C_0 g_a(x), \\
\abs{\widehat{u}(x,0)} = \abs{u\(x,-\frac{1}{8}\)}
=   & \; C_1 g_b(x), \quad{\text{but}}\\
\abs{u\(x,\frac{1}{16}+\frac{\tau}{4}\)}
=   & \; C_2 g_{\omega(a,b)}(x),
\end{aligned}
\end{equation*}
for some $C_0$, $C_1$, and $C_2$ which possibly depends on $a$, and $b$.

We now proceed to prove point-wise Gaussian bound for $u(x,t)$.
To do so, write
$$
u(x,0) = \sum_{k=0}^\infty \<u(\cdot,0), h_k\> h_k(\sqrt{2\pi}x),
$$ 
then
$$
u(x,t) = \sum_{k=0}^\infty e^{(2k+1)\pi it} \<u(\cdot,0), h_k\> h_k(\sqrt{2\pi}x).
$$ 
Assume $u(x,0) \in E^1(a,b)$. The estimate 
(\ref{E:Chaurasia-estimate}) will be crucial for us, which also provides us that 
$$
\<u(\cdot,0), h_k\> \lesssim A^{k/2}.
$$

If we use Cauchy-Schwartz inequality and the Mehler's formula,
as used in \cite{Folland-Sitaram} and \cite[Theorem 3.1]{Vemuri2008hermite},
instead of Theorem \ref{T:Radchenko-Ramos}, then here as well 
we obtain an $\epsilon$-loss in comparison with the desired estimates.
To see that, let $B\in (A,1)$. Then 
\begin{equation*}
\begin{aligned}
\abs{u(x,t)}
\le & \; \(\sum_{k=0}^\infty B^{-k} \abs{\<u(\cdot,0), h_k\>}^2\)^{1/2}
         \(\sum_{k=0}^\infty B^k \abs{h_k(\sqrt{2\pi}x)}^2\)^{1/2} \\
\le & \; C(A,B) e^{-\pi\frac{1-B}{1+B} x^2}.
\end{aligned}
\end{equation*}
Also,
$$
\abs{\widehat{u}(x,t)}
= \abs{u(x, t - 1/8))}
\le  C(A,B) e^{-\pi\frac{1-B}{1+B} x^2}.
$$
Thus $u(x,t)\in E\(\frac{1-B}{1+B},\frac{1-B}{1+B}\)$.

In order to get the sharp estimates, we use the full strength of estimate
(\ref{E:Chaurasia-estimate}) and Theorem \ref{T:Radchenko-Ramos}. 
That is, we have
\begin{equation*}
\begin{aligned}
\abs{u(x,t)} \lesssim&\;
\sum_{k=0}^\infty  \abs{\<u(\cdot,0), h_k\>} \abs{h_k(\sqrt{2\pi}x)}\\
\lesssim&\; \sum_{k=0}^\infty  k^{-1/4} A^{k/2} \abs{h_k(\sqrt{2\pi}x)}\\
=&\; \sum_{k=0}^\infty  k^{-1/4} e^{-k\frac{1}{2}\log\frac{1}{A}}
 \abs{h_k(\sqrt{2\pi}x)}.
\end{aligned}
\end{equation*}
Applying Theorem \ref{T:Radchenko-Ramos} for $\kappa=1$, $\beta=1/4$, and
$y=\tanh\(\frac{1}{2}\log \frac{1}{A}\)$ gives us that
\begin{equation*}
\abs{u(x,t)} \lesssim_{a,b} e^{-\pi \tanh\(\frac{1}{2}\log \frac{1}{A}\) x^2}.
\end{equation*}
Observe that
\begin{equation*}
\tanh\(\frac{1}{2}\log \frac{1}{A}\) = \frac{1-A}{1+A} = \omega(a,b).
\end{equation*}
Therefore, we conclude that
\begin{equation*}
\abs{u(x,t)} \lesssim_{a,b} e^{-\pi \omega(a,b) x^2}.
\end{equation*}
Also, by the relation $\abs{\widehat{u}(x,t)}
= \abs{u(x, t - 1/8))}$, we get the same Gaussian bound for 
$\widehat{u}(x,t)$ as well. This finishes the proof.

\qed

Now, we first point out that it is natural to expect a better time-dependent
Gaussian bounds for $u(x,t)$ and $\widehat{u}(\xi,t)$ 
when $u(x,0)\in E^1(a, b)$, for $a, b>0$, and $ab<1$. 
For $\mu$ to be a finite measure supported on the positive real line, let 
$\phi(x)=\mathcal{L}\mu(\pi\abs{x}^2)$, where 
$$
\mathcal{L}\mu(s) = \int_0^{\infty} e^{-st} d\mu(t).
$$ 
We show that 
\begin{equation}\label{E:time-evol}
\Phi(x,t) = e^{-itH}\phi(x)\in E^1(\eta(t),\eta(t)),
\end{equation}
for some $\eta(t)\ge \min\{a,b\}\ge \omega(a,b)$, when $\phi\in E^1(a,b)$
for $a, b>0$, and $ab<1$.
We proceed with the following lemma.
\begin{lem}\label{L:sigma-less-a}
Let $a, b>0$, and $ab<1$. Let $\omega(a,b)$ be as defined in equation 
(\ref{E:v-function}). Then
$$
\omega(a,b)\le \min\{a,b\}.
$$
\end{lem}
\begin{proof}
Define $J:(0,1]\rightarrow \R$ by $J(x) = x+\frac{1}{x}$.
It is direct to see that the function $J$ is strictly decreasing on $(0,1)$.
We note from equation $(\ref{E:criteria})$ that $\omega(a,b)\le1$.
Therefore, equation (\ref{E:v-function}) gives us that 
\begin{equation*}
J(\omega(a,b)) = \frac{1}{a}+\frac{1}{b}.
\end{equation*}
Let $a\le b$. Then, due to the fact that $ab<1$, we see that 
\begin{equation*}
J(\omega(a,b)) \ge a+\frac{1}{a}= J(a).
\end{equation*}
Taking into account that $J$ is a strictly decreasing function, we finish the proof.
\end{proof}
For $\beta\in\R$, define the {\it Fractional Fourier transform} 
of $\phi$ of order $\beta$ as defined in \cite[equation 4.2]{Kulikov.et.al-2024}
\begin{equation*}
\mathcal{F}_{\beta} \phi(x) = 
\frac{e^{i(\theta(\beta)\pi/2-\beta/2)}}{\sqrt{\abs{\sin(\beta)}}}
e^{i\pi x^2\cot(\beta)}
\int_{\R} e^{-2\pi i(xy\csc(\beta)-y^2\cot(\beta)/2)}\phi(y) dy,
\end{equation*}
where $\theta(\beta)= \mathrm{sgn}(\sin(\beta))$. Then 
\begin{equation*}
\Phi(x,t) = e^{2\pi i t} \mathcal{F}_{-4\pi t}\phi (x).
\end{equation*}
Set $m=\min\{a,b\}$. Then, $\phi\in E^1(a,b)$ implies 
$\phi\in E^1(m,m)$. With this, we see that 
the arguments given in \cite{Kulikov.et.al-2024}
gives us
\begin{equation*}
\supp\(\mu\) \subset \[m, \frac{1}{m}\].
\end{equation*}
Observe that, for $\lambda>0$
\begin{equation}\label{E:Fractional}
\abs{\mathcal{F}_{\beta}g_{\lambda}(x)} = 
C_{\lambda, \beta}\, e^{-\pi\frac{\lambda\csc^2(\beta)}{\lambda^2+\cot^2(\beta)}
\abs{x}^2},
\end{equation}
for some constant $C_{\lambda, \beta}$, which depends on $\lambda$, and $\beta$.
Thus, it follows that
\begin{equation*}
\begin{aligned}
\abs{\Phi(x,t)} \lesssim_{t}&\; \int_m^{1/m} 
e^{-\pi \frac{\lambda \csc^2(4\pi t)}{\lambda^2+\cot^2(4\pi t)} \abs{x}^2}
\abs{d\mu}(\lambda)\\
\lesssim_{t,\mu}&\;  \max\Big\{
e^{-\pi \frac{m \csc^2(4\pi t)}{m^2+\cot^2(4\pi t)} \abs{x}^2},
e^{-\pi \frac{m \csc^2(4\pi t)}{1+m\cot^2(4\pi t)} \abs{x}^2}\Big\}.
\end{aligned}
\end{equation*}
Since
$$
\min\Big\{\frac{m \csc^2(4\pi t)}{m^2+\cot^2(4\pi t)}, 
\frac{m \csc^2(4\pi t)}{1+m\cot^2(4\pi t)}\Big\} \ge \min\{a,b\},
$$
with equality if and only if $t\in \frac{1}{8}\Z$.
Therefore from Lemma \ref{L:sigma-less-a}, there exists $\eta(t)$
with $\eta(t)\ge \min\{a,b\}\ge \omega(a,b)$ such that (\ref{E:time-evol})
holds.

The preceding analysis and the proof of Theorem \ref{T:Asym-evolution}
 suggest the validity of the
 following conjecture, which may be viewed as a slight generalization of 
 \cite[Conjecture 3.4]{Kulikov.et.al-2024}.
\begin{conj}
Let $a,b>0$, and $ab<1$. 
Let $\Omega_{a,b}$ be as in equation \eqref{E:omega-lambda}.
 Let $u(x,t)$ be the solution of the Cauchy problem 
(\ref{E:HSE}) with the initial data $u_0$.
If $u_0\in E^1(a,b)$, then
$$
u(x,t) \in E^1\(\Omega_{a,b}(s),\Omega_{a,b}(s)\),
$$
for $s=-\frac{\tan4\pi t}{4\pi}$.
\end{conj}
This conjecture is known to hold for a certain class of functions; see
 \cite{Kulikov.et.al-2024}. 
 We partially resolve the above conjecture for the class of 
 functions described in Theorem \ref{T:time-bound}. 
 The proof of this theorem occupies the next subsection.

\subsection{Proof of Theorem \ref{T:time-bound}.}
For $r>0$, denote $\mathbb{S}^1_{r}$ as the circle of radius $r$ in $\C$.
The Laplace transform of a finite measure $\nu$ supported on $\mathbb{S}^1_{r}$
can be defined in the following way:
$$
\mathcal{L}\nu (t) =
\int_{\mathbb{S}^1_{r}} e^{-tz} d\nu(z).
$$
We proceed with the following Lemma.
\begin{lem}\label{L:supp-lem}
Let $r>0$. Let $\nu_r$ be a finite measure supported on $\mathbb{S}^1_{r}$
and let $-r\notin \supp(\nu_r)$. 
If $\abs{\mathcal{L}\nu_r (t)} \lesssim e^{-ct}$ for some $c>0$,
then 
\begin{equation*}
\supp(\nu_r) \subset \mathbb{S}^1_{r} \cap 
\{z \st \Re{(z)}\ge c\}.
\end{equation*}
\end{lem}
\begin{proof}
We note that
\begin{equation*}
\begin{aligned}
\mathcal{L}\nu_r (t) = &\;
\int_{\mathbb{S}^1_{r}} e^{-tz} d\nu_1\(\frac{z}{r}\)\\
=&\;
\int_{\mathbb{S}^1_{1}} e^{-rtz} d\nu_1(z)\\
=&\;
\mathcal{L}\nu_1 (rt).
\end{aligned}
\end{equation*}
Therefore, we have 
$$
\abs{\mathcal{L}\nu_1 (t)} \lesssim e^{-\frac{c}{r}t}.
$$
Also, $-1\notin \supp{(\nu_1)}$. Thus, from \cite[Lemma 4.2]{Kulikov.et.al-2024},
we obtain
\begin{equation*}
\supp(\nu_1) \subset \mathbb{S}^1_{1} \cap 
\Big\{z \st \Re{(z)}\ge \frac{c}{r}\Big\}.
\end{equation*}
Hence, we have the desired result. 
\end{proof}
Now, set $A_1 = 
\Big\{{\sqrt{\frac{a}{b}}} e^{ i\theta}\subset
\mathbb{S}^1_{\sqrt{\frac{a}{b}}}\st \theta\in [0,\frac{\pi}{2})\Big\}$
and 
$A_2 = 
\Big\{{\sqrt{\frac{b}{a}}} e^{ i\theta}\subset
\mathbb{S}^1_{\sqrt{\frac{b}{a}}}\st \theta\in (-\frac{\pi}{2}, 0]\Big\}$.
Let $\mu$ be a measure supported on $L_{a,b}$.
Define a map $\Gamma : L_{a,b}\rightarrow A_1$ by
\begin{equation*}
\Gamma(v,w) = \sqrt{\frac{a}{b}}\, e^{i\arccos{\sqrt{vw}}}.
\end{equation*}
This map allows us to define the push-forward measure $\nu = \Gamma_*(\mu)$ 
of $\mu$ on the circle $\mathbb{S}^1_{\sqrt{\frac{a}{b}}}$.
With this, we can rewrite $u_0$ as 
\begin{equation}\label{E:u-L}
u_0(x) = \int_{\mathbb{S}^1_{\sqrt{\frac{a}{b}}}} e^{-z\pi\abs{x}^2} d\nu(z)
= \mathcal{L}\nu (\pi \abs{x}^2).
\end{equation}
Also, we see from the given expression of $u_0$ that
\begin{equation*}
\widehat{u_0}(\xi) = C_{a,b}\int_{L_{a,b}} \bar{G_{w,v}(\xi)} d\mu(v,w),
\end{equation*}
for some constant $C_{a,b}$ which depends on $a$ and $b$. 
As done above, we use the map $\widetilde{\Gamma} : L_{a,b}
\rightarrow A_2$ defined by $\widetilde{\Gamma}= \frac{b}{a} \bar{\Gamma}$ to 
write
\begin{equation*}
\widehat{u_0}(\xi) = 
\int_{\mathbb{S}^1_{\sqrt{\frac{a}{b}}}} e^{-\frac{b^2}{a^2}
z\pi\abs{\xi}^2} d\nu(\bar{z}).
\end{equation*}
Thus
\begin{equation}\label{E:F-u-L}
\widehat{u_0}(\frac{a}{b}\xi) = \mathcal{L}\nu (\pi \abs{\xi}^2).
\end{equation}
Since, $u_0\in E^1(a,b)$, therefore, 
from equation (\ref{E:u-L}) and (\ref{E:F-u-L}), we obtain
\begin{equation*}
\abs{\mathcal{L}\nu(t)} \lesssim_{a,b} e^{-at}.
\end{equation*}
Hence, from Lemma \ref{L:supp-lem}, we get 
$\supp(\nu) \subset \mathbb{S}^1_{\sqrt{\frac{a}{b}}} \cap 
\{z \st \Re{(z)}\ge a\}$. Then 
\begin{equation}\label{L:mu-supp}
\supp(\mu) \subset \Big\{(v,w)\in L_{a,b} \st  a\le v \le \sqrt{\frac{a}{b}}\Big\}.
\end{equation}
A direct computation using equation (\ref{E:Fractional}) provides us
\begin{equation*}
\abs{\mathcal{F}_{-4\pi t} G_{w,v} (\xi)} = C_{t,v,w} 
e^{-\pi\Omega_{v,w}\(-\frac{\tan4\pi t}{4\pi}\)\abs{\xi}^2},
\end{equation*}
where $C_{t,v,w}$ is some constant, which depends on $t$, $v$, and $w$.
By using this formula, the relation between the solutions of the problem (\ref{E:HSE}) 
and the fractional Fourier transform, and the given integral representation of 
$u_0$, we obtain 
\begin{equation*}
\begin{aligned}
\abs{\widehat{u}(\xi,t)} =&\;  \int_{\supp(\mu)} 
e^{-\pi\Omega_{v,w}\(-\frac{\tan4\pi t}{4\pi}\)\abs{\xi}^2} C_{t,v,w}\,d\mu(v,w)\\
\lesssim_{t,a,b,\mu}&\; \max\Big\{
e^{-\pi\Omega_{a,b}\(-\frac{\tan4\pi t}{4\pi}\)\abs{\xi}^2},
e^{-\pi\Lambda_{a,b}\(-\frac{\tan4\pi t}{4\pi}\)\abs{\xi}^2}\Big\}.
\end{aligned}
\end{equation*}
We get the same bound for $u(x,t)$ by using the relation 
$\abs{u(x,t)} = \abs{\widehat{u}(x,t+1/8)}$. 
Observe that
\begin{equation*}
\Lambda_{a,b} (t) \le \Omega_{a,b}(t) \quad t\in [l_-(a,b), l_+(a,b)].
\end{equation*}
With this observation and the estimates for $u(x,t)$ and
$\widehat{u}(\xi,t)$ above, we conclude the proof. 
\qed
\section*{Conflict of Interest}
The authors declare no conflicts of interest.

\section*{Funding}
This research received no external funding.

\bibliographystyle{amsplain}
\bibliography{v0-GDHLO}

\end{document}